%% file: main.tex
\documentclass[final,letterpaper]{article}
\usepackage{times}
\usepackage{amsthm}
\usepackage{amsmath}
\usepackage{amssymb}
\usepackage{xspace}
\usepackage{xcolor}
\usepackage{bm}

\newtheorem{theorem}{Theorem}
\newtheorem{proposition}[theorem]{Proposition}

\theoremstyle{definition}

\newtheorem{remark}[theorem]{Remark}
\newtheorem{example}[theorem]{Example}
\newtheorem{algorithm}[theorem]{Algorithm}

\title{A class of regularization schemes for linear ill-posed problems in Banach spaces under low order source conditions}

\author{Robert Plato\footnotemark[1]}

\input macros.txt

\begin{document}
\date{}
\maketitle


\renewcommand{\thefootnote}{\fnsymbol{footnote}}
\footnotetext[1]{Department of Mathematics, University of Siegen,
Walter-Flex-Str.~3, 57068 Siegen, Germany}

\newcounter{enumcount}
\newcounter{ccount}
\renewcommand{\theccount}{\arabic{ccount}}
\setcounter{ccount}{-1}
\newcounter{enumcountroman}
\renewcommand{\theenumcount}{(\alph{enumcount})}
\bibliographystyle{plain}
%
%
\begin{abstract}
In this work we consider, in a Banach space framework, the regularization of linear ill-posed problems. Our focus is on the recovery of solutions that have a logarithmic source representation. Such cases typically occur in exponentially ill-posed problems like the backwards heat equation. The mathematical framework on logarithms of operators is provided, and convergence rates for a class of parametric regularization schemes are deduced. This class includes the iterated version of Lavrentiev's method and the method of the abstract Cauchy problem. The presentation includes both a priori and a posteriori parameter choice strategies. Finally, we present an example for a logarithmic source representable function with respect to the integration operator.
\end{abstract}


%


%
\section{Introduction}
\label{intro}
In this paper, we consider linear equations
\begin{align}
\label{eq:maineq}
Au = \fst,
\end{align}
where $ A:\ix \to \ix $ is a bounded linear operator in a Banach space $ \ix $ with a non-closed range $ \R(A) $. Moreover we suppose that
the \rhs $ \fst $ belongs to the range $ \R(A) $
of $ A $ and is given approximately
by $ \fdelta \in \ix $ as follows, \\
$$
\fst \in \R(A),
\quad 
\fdelta \in \ix, \qquad
\norm{ \fst - \fdelta } \le \delta,
$$
where $ \delta \ge 0 $ is a given noise level. The fact that $ A $ has a non-closed range implies that equation \refeq{maineq} is ill-posed, \ie the solution depends discontinuously on the data $ \fst $. We assume here for the moment that the operator $ \A $ is one-to-one, which means that solutions of equations $ \A u = f $ are unique for each $ f \in \R(A) $.
Some examples are considered in Example \ref{th:postype-examples} below.
The main subject of the present paper is the consideration of equations~\refeq{maineq} having solutions with a
logarithmic source-wise representation, \ie $ u \in \domain(\logpowa) $ for some integer $ \nu \ge 1 $. This in fact is a typical situation for exponentially ill-posed problems, e.g., the backwards heat equation. In that case, the operator $ A $ is of the form 
$ A = \exp(-tB) $ with some $ t > 0 $ and some differential operator $ B $. We then have
$ \log A = -tB $ and thus $ \domain(\logpowa) = \domain(B^\mynu)$, so that exactly the elements of $ \domain(\logpowa) $ have certain differentiability properties.

Equations with logarithmic source-wise representable solutions are considered in several works, see \eg 
\cite{Mair[94],Tautenhahn[98],Hohage[00],Mathe_Pereverzev[03],Hofmann_Mathe[07],König_Werner_Hohage[16],Kindermann[21],Klinkhammer_Plato[23],Klinkhammer[24],Plato_Hofmann[22],Hofmann_Klinkhammer_Plato[23]}. The latter two papers include a Banach space setting, but the considerations are restricted to Tikhonov regularization. In the present work, our focus is on other regularization schemes in a Banach space framework, \eg \lavmet. The considerations are restricted to linear problems, however.

Our consideration are not restricted to genuine logarithmic source-wise representations. More generally, we will consider mixed source conditions of the form
$ \myu = A^\pp v $ for some $ v \in \domain(\logpowa) $,
with some $ p \ge 0, \ \nu \in \naza $. Those extentions are obtained without much additional efforts and maybe of independent interest.

The present work can be considered as an extention of the work
\cite{Plato[96]}, where a similar setting is considered for power-type source-wise representable solutions $ \myu \in \R(A^\pp) $ for some $ p > 0 $. For related work
on power-type source-wise representable solutions for linear operators in Banach spaces, we refer to \cite{Nevanlinna[93]}.


The outline of the remainder of the paper is as follows. The following section provides some basic facts on \postype operators and its fractional powers,
which serve as a basic ingredient in the analysis. Moreover, a class of regularization schemes is introduced in that section, which includes \lavmet and its iterated version. Section \ref{log_A} introduces the logarithm of a bounded \postype operator and some of its properties with impact to regularization. In Section~\ref{parameterchoice},
for the considered class of regularization schemes, we present convergence rates 
under mixed source conditions. This is done
both for appropriate a priori- as well as a posteriori parameter choice strategies, where the latter rule serves as a discrepancy principle.
As a by-product, we obtain a logarithmic interpolation inequality which may be of independent interest. 
In Section~\ref{loworder_example}, we present an example of a function that satisfies a \losoc.
 In an appendix, some useful results for a special class of real-valued, logarithm-based functions are provided that can be utilized for the derivation of error estimates under mixed source conditions.
\section{\Postype operators, fractional powers, and regularization operators}
\subsection{\Postype operators}
\label{postype_operators}
Let $ \ix $ be a Banach space and $ \A: \ix \to \ix $ be a bounded linear operator of \postype, i.e.,
\begin{align}
& \A + \para I: \ix \to \ix \ \textup{one-to-one and onto},
\qquad \norm{(\A + \para I)^{-1}} \le \frac{\kapst}{\para}, \quad \para > 0,
\label{eq:postype}
\end{align}
for some finite constant $ \kappa_* > 0 $.
Property \refeq{postype} provides the basic ingredient both for defining fractional powers and the logarithm of the operator $ A $, respectively. It is also needed for
\lavmet and its iterated version.
\begin{example}
\label{th:postype-examples}
\begin{myenumerate}
\item
For every bounded linear operator which is self-adjoint and positive semi-definite on a Hilbert space $ \ix $, condition \refeq{postype} is satisfied with $ \kapst = 1 $.

\item
\label{it:postype-examples-b}
Consider the integration operator 
\begin{align*}
\klaf{Ju}(x) = \ints{0}{x}{ u(\myxi)}{ d\myxi } \qquad (0 \le x \le \mya, \quad u \in \ix).
\end{align*}
It defines a bounded linear operator $ J: \ix \to \ix $ \ofpostype in each of the following cases:
$ \ix = C[0,\mya] := \inset{u: [0,\mya] \to \reza \mid u \textup{ continuous } }, \ 
\ix = C_0[0,\mya] := \inset{u \in C[0,\mya]:u(0) = 0 }$, and
$ \ix=L^r(0,\mya)$ for $ 1 \le r \le \infty $, \cf, \eg \cite{Plato[95],Plato[97.2]}.

\item
Consider the Abel integration operator 
\begin{align}
\klaf{J_\myalpha u}(x) = \mfrac{1}{ \Gamma\kla{\myalpha} }
\ints{0}{x}
{ \mfrac{ u\kla{\myxi} }{ \kla{x-\myxi}^{1-\myalpha}} }
{ d \myxi }
\for 0 \le x \le \mya, \quad u \in \ix,
\label{eq:frac-int=abel}
\end{align}
where $ \myalpha > 0 $, and $ \Gamma $ denotes Euler's gamma function. Moreover,
$ \ix $ denotes one of the spaces considered
in part \ref{it:postype-examples-b} of this example.
%
For each $ 0 < \myalpha < 1 $ and each such space $ \ix $, the Abel integral operator  
$ J_\myalpha: \ix \to \ix $ is of \postype, respectively, \cf \cite{Plato[95],Plato[97.2]}.
\remarkend
\end{myenumerate}
\end{example}
\subsection{Fractional powers}
Fractional powers of a \postype operator $ A: \ix \to \ix $ in a Banach space $ \ix $ may be utilized to introduce an intermediate degree of smoothness of solutions of $ A \myu = f $. In the present paper, they will be also used to introduce the logarithm of the operator $ A $.

Fractional powers of a bounded \postype
operator $ A: \ix \to \ix $  may be defined as follows, \cf
\cite{Balakrishnan[59],Balakrishnan[60]} or \cite[Proposition 3.1.12]{Haase[06]}:
\begin{myenumerate_indent}
\item
For $ 0 < \pp < 1 $, the fractional power $ \A^\pp : \ix \to \ix $ is defined by the improper Bochner integral
\begin{align}
\A^\pp u \defeq \mfrac{\sin \pi \pp }{\pi}
\ints{0}{\infty}{s^{\pp-1} (\A + sI)^{-1} \A u }{ ds}
\for u \in \ix.
\label{eq:frac-power}
\end{align}
\item
For arbitrary values $ \pp > 0 $,
the bounded linear operator $  \A^\pp : \ix \to \ix $ is defined by
\begin{align}
\label{eq:frac-power-b}
\A^\pp  \defeq \A^{\pp - \lfloor \pp \rfloor} \A^{\lfloor \pp \rfloor},
\end{align}
where $ \lfloor \pp \rfloor $ denotes the largest integer $ \le \pp $, and moreover we use the notation $ \A^0 = I $.
\end{myenumerate_indent}

\begin{example}
\begin{myenumerate}
\item
For a bounded linear operator which is self-adjoint and positive \semidefinite on a Hilbert space $ \ix $, representations \refeq{frac-power}, \refeq{frac-power-b} on one side and the standard definition of fractional powers based on spectral calculus result in the same operators.

\item The fractional powers of the integration operator $ J $ coincide with the Abel operators, \ie $ J^\pp = J_\pp $ for $ \pp > 0 $. This holds in each of the spaces $ \ix $ considered in Example~\ref{th:postype-examples} \ref{it:postype-examples-b}. 
For details, \cf \eg \cite{Plato[95],Plato[97.2]}.
\remarkend

\end{myenumerate}
\end{example}
We note that 
\begin{align}
\label{eq:semigroup-property}
\A^{\pp+\qq} = \A^{\pp} \A^{\qq} \for \pp, \qq \ge 0, \quad
\A^\pp \myu \to \myu \as \pp \downarrow 0 \for \myu \in \overline{\R(\A)},
\end{align}
cf.~\cite[Propositions 3.1.1 and 3.1.15]{Haase[06]}).
The statements in \refeq{semigroup-property}
imply that the operators $ \apsegr $ define a \cosegr on $ \overline{\R(\A)} $. This fact has great impact throughout the paper.

From \refeq{semigroup-property} and the definition
\refeq{frac-power}, it moreover follows that
\begin{align} \label{eq:chain1}
\R(\A^{\pp_2}) \subset\R(\A^{\pp_1}) \subset \overline{\R(\A)} \textup{ for all } \; 0 < \pp_1 < \pp_2 < \infty.
\end{align}
In what follows, we shall need the interpolation inequality for \fracpowers of operators, see, e.g.,
\cite{Komatsu[66]} or
\cite[Proposition 6.6.4]{Haase[06]}:
for each pair of real numbers \linebreak $ 0 < \pp < \qq $,
there exists some finite constant $ \ccdefref{interpol} = \ccref{interpol} (\pp,\qq) > 0 $ such that
\begin{align}
\norm{ \A^\pp u }
\le \ccref{interpol} \norm{\A^\qq u}^{\frac{\pp}{\qq}} \norm{ u }^{1-\frac{\pp}{\qq}}
\for u \in \ix.
\label{eq:interpol2}
\end{align}
For $ 0 < \pp < 1 = \qq $, the constant may be chosen as $ \ccref{interpol} = 2(\kappa_* + 1)$, cf., e.g., \cite[Corollary 1.1.19]{Plato[95]}.

Throughout the paper, we assume that the operator $ \A $ is one-to-one,
with a possibly non-dense range in $ \ix $,
and in addition the inverse $ \A^{-1} $ is an unbounded operator.
%
\subsection{A class of regularization schemes}
In what follows, we consider bounded linear regularization operators associated with a bounded \postype
operator $ A: \ix \to \ix $,
\begin{align}
R_\para: \ix \to \ix	\for \para > 0
\label{eq:rbeta}
\end{align}
and its companion operators
\begin{align}
S_\para \defeq  I - R_\para \A \for \para > 0.
\label{eq:sbeta}
\end{align}
We assume that the following conditions are satisfied:
\begin{align}
\norm{ R_\para } & \le \mfrac{\cst}{\para} \for \para > 0,
\label{eq:wachstum} \\
\norm{ S_\para \A^\pp } & \le \cp \para^{\pp} \for \para > 0
\qquad (0 \le \pp \le \pp_0),
\label{eq:abfall} \\
R_\para \A & = \A R_\para \for \para > 0,
\label{eq:commute}
\end{align}
where $ 0 < \pp_0 \le \infty $ denotes the possibly infinite saturation, and $ \cst $ and $ \cp $ denote finite constants. We assume that $ \cp $ is bounded on bounded intervals for  $ \pp $. In the case $ \ppz = \infty $,
condition \refeq{abfall} is considered for $ \pp < \infty $.
Note that condition \refeq{abfall} in particular implies 
$ S_\para u \to 0 $ as $ \alpha \to 0 $ for each $ u \in \ix $ and thus
$ R_\para f \to A^{-1} f $ as $ \alpha \to 0 $ for each $ f \in \R(A) $, if
$ A $ is a one-to-one operator with dense range.

We are now in a position to introduce \emph{regularizing elements}
%
\begin{align}
\upardel \defeq \ubar - R_\para( \A \ubar - \fdelta)
\for \para > 0,
\label{eq:upardel-def}
\end{align}
where $ \ubar \in \ix $ denotes an initial guess to the sought-for-solution of the equation $ A\myu = \fst $.
%
\begin{example} We first consider \lavmetit, where 
$ m \ge 1 $ denotes a fixed integer. For $ \fdelta \in \ix $ and $ v_0 = \ubar \in \ix $, the element $ \upardel \in \ix $
is given by
\begin{align*}
(\A + \para I)v_n & =  \para v_{n-1} + \fdelta \for n = 1,2,\ldots,m, \qquad
\upardel \defeq v_m.
\end{align*}
This method is of the form \refeq{upardel-def} with
%
$$ R_\para = \para^{-1} \sum_{j=1}^{m} \para^j (\A + \para I)^{-j},  $$
and the companion operator is given by $ S_\para = \para^m (\A + \para I)^{-m} $.
For $ m = 1 $, this gives Lavrentiev's classical regularization method,
$ R_\para = (\A + \para I)^{-1} $.
For \lavmetit, the conditions \refeq{wachstum}--\refeq{commute} are satisfied with
$ p_0 = m $. In fact,
for any integer $ 0 \le p \le m $,
estimate \refeq{abfall} holds with constant $ \cp = (\kappa_*+1)^m $, see
\cite[Lemma 1.1.8]{Plato[95]}.
From this intermediate result and the interpolation inequality~\refeq{interpol2}, inequality \refeq{abfall} then follows
for non-integer values $ 0 < p < m $, with constant $ \cp = 2(\kappa_*+1)^{p+1} $.
\remarkend
\end{example}
\begin{example}
Let $ A: \ix \to \ix $ be a bounded linear operator on a complex Banach space~$ \ix $.
In the present example, we assume that $ A $ satisfies a sectorial condition stronger than 
\refeq{postype} considered for \postype operators. More precisely, we assume that
for each 
$ \lambda = r e^{\ii \varphi} $, with
$ r > 0, \ \modul{ \varphi } \le \frac{\pi}{2} + \varepsilon $ with some $ \varepsilon > 0 $, 
the operator $ \A + \lambda I: \ix \to \ix $ is one-to-one and onto, and its inverse operator satisfies $  \norm{(\A + \lambda I)^{-1}} \le \frac{\kapst}{\modul{\lambda}} $, where $ \kapst > 0 $ denotes a finite constant.
Examples of this type are given by the Abel integral operators $ J_\myalpha, \ 0 < \myalpha < 1 $, considered in
Example \ref{th:postype-examples}. On the other hand, the simple integration operator considered in the same example does not satisfy such a strong sectorial condition.
We refer to \cite{Plato[95],Plato[97.2]} for more details.

For any operator $ A $ satisfying such a strong sectorial condition and any initial approximation $ \ubar \in \ix $, the abstract Cauchy problem,
\begin{align}
\kla{\udelta}^\prime\kla{t} + A\udelta\kla{t} = \fdelta \for t > 0,
\quad 
\udelta \kla{0} = \ubar,
\label{eq:acp12}
\end{align}
has a unique solution $ \udelta\kla{t}, \ t \ge 0 $.
The regularizing elements are then defined as
\begin{align*}
\upardel \defeq \udelta\kla{t} \for \alpha = \tfrac{1}{t}, \ t > 0,
\label{eq:acp3}
\end{align*}
and they in fact can be written in the form \refeq{upardel-def} with
\begin{equation*}
R_\para f = \ints{0}{t}{ e^{-s A} f }{ds } \for f \in \ix,
\ t = \tfrac{1}{\para}, \ \para > 0.
\end{equation*}
The companion operator is given by $ S_\para = e^{-t A} $ for $ t = \frac{1}{\para}, \ \para > 0 $.
For this method, the conditions \refeq{wachstum}--\refeq{commute} are satisfied with
qualification $ p_0 = \infty $. See \cite[Theorems~1.2.10 and 2.1.4]{Plato[95]} for more details.
\
\remarkend
\end{example}
\section{The logarithm $\bm{\log \A}$}
\label{log_A}
\subsection{Definition of $ \bm{\log \A} $, first properties}
We next consider logarithmic source conditions, occasionally called low order smoothness. For this, we need to introduce the logarithm $ \log A $ of a bounded \postype operator $ \A $. For selfadjoint operators in Hilbert spaces, this can be done by spectral analysis, and we refer in this context for example to \cite{Hohage[00],Tautenhahn[98]}.
In Banach spaces, the operator $ \log \A $ may be defined as the infinitesimal generator of the \cosegr of operators $ \apsegr $ considered on $ \RA $:
\begin{align}
\label{eq:log-def}
(\log \A) \myu & = \lim_{p \downarrow 0} \tfrac{1}{p}\kla{\A^p \myu-\myu},
\quad \myu  \in \domain(\log \A),
\end{align}
where
\begin{align*}
\domain(\log\A) & =
\inset{ \myu \in \ix : \lim_{p \downarrow 0} \tfrac{1}{p}\kla{\A^p \myu- \myu}
\ \textup{exists} },
\end{align*}
cf.,~e.g.,~\cite{Nollau[69],Fury[20]} or \cite[Proposition 3.5.3]{Haase[06]}.
For Tikhonov regularization with oversmoothing penalty term to solve nonlinear ill-posed problems in Banach scales,
the definition \refeq{log-def} has already been utilized in the two papers
\cite{Plato_Hofmann[22],Hofmann_Klinkhammer_Plato[23]}
to describe \los of solutions.
%
In the present paper, \los of an element $ \myu \in \ix $ by definition means
$ \myu \in \domain(\logpowa) $ for some $ \mynu \in \naza $.
Note that we obviously have $ \domain(\logpowa) \subset \RA $. In addition, $ \R(\A^\pp) \subset \domain(\logpowa)$ is valid for arbitrarily small $p>0$,
which follows from \cite[Satz~1]{Nollau[69]}. Summarizing the above notes, we have a chain of inclusions in $\ix$ as
\begin{equation} \label{eq:chain2}
\R(\A^p) \subset 
\domain(\logpowa[\mynu_2])
\subset 
\domain(\logpowa[\mynu_1])
\subset \overline{\R(\A)} 
\quad \textup{for} \ \mynu_1, \mynu_2 \in \naza, \ \mynu_1 \le \mynu_2, \  \pp>0.
\end{equation}
This in particular means that any H\"older-type smoothness is stronger than low order smoothness, a property that is already well known in the Hilbert space setting. 
\subsection{Auxiliary results for $ \bm{\log \A} $}
Throughout this subsection, we consider a bounded linear operator 
$ A :\ix \to \ix $ of \postype on a Banach space $ \ix $.
\subsubsection{Introductory remarks}
Below, for elements $ \myu \in \ix $ we consider mixed smoothness conditions of the form
\begin{align}
\label{eq:smooth-cond}
\myu = A^\pp v \textup{ for some } v \in \domain(\logpowa), 
\quad 0 \le p < \pp_0, \quad \nu \in \naza.
\end{align}
We define a source element for $ \myu $ as follows,
\begin{align}
\label{eq:log-source}
w = (\lambda I - \loga)^\nu v, \quad \textup{ with } \lambda > \mylog \norm{\A} \textup{ arbitrary but fixed}. 
\end{align}
Here and in what follows, for any real number $ x > 0 $, we denote by $ \log x $ the natural logarithm of $ x $.
%
The lower bound for $ \lambda $ considered in \refeq{log-source}
guarantees that
the operator $ \lambda I - \loga $ has a bounded inverse,
\cf the proof of the following theorem.
This in fact is the reason for considering source elements $ w $ in the shifted form \refeq{log-source}. It allows 
to rewrite smoothness condition \refeq{smooth-cond} in the form 
\begin{align}
\myu = A^\pp (\lambda I - \loga)^{-\nu} w,
\label{eq:mixed_smoothness}
\end{align}
so that the element $ w $ indeed can be conceived as a source element for $ \myu $.
In \refeq{mixed_smoothness}, the operator $ (\lambda I - \loga)^{-\nu} $ denotes the inverse operator of $ (\lambda I - \loga)^{\nu} $.
\begin{remark}
\begin{myenumerate}
\item
Note that the operators $ \lambda I - \loga $ and $ \loga $ have the same domains of definition, so the shift of $ \loga $ considered in \refeq{log-source} does not really change the mixed smoothness condition \refeq{smooth-cond}.

\item
In the case $ \norm{A} < 1 $,
the source representation \refeq{log-source} can be obtained without any shift, \ie
$ w = \logpowa v $, since we may choose $ \lambda = 0 $ then.
Note that the norm condition on $ A $ can always be obtained by rescaling the equation $ A u = \fst $ to $ a A u = a \fst $, with $ a > 0 $ sufficiently small such that the inequality $ a \norm{A} < 1 $ is satisfied. 
Note also that such a rescaling does not change the smoothness condition 
\refeq{smooth-cond} since we have $ (aA)^\pp = a^\pp A^\pp $ and 
$ \log (aA) = (\mylog a) I + \loga $, which easily follows from the respective definitions.
\remarkend
\end{myenumerate}
\end{remark}

\subsubsection{Decay of $ \norm{ S_\para \myu  } $}
Throughout this subsection, we consider regularizing operators
\refeq{rbeta}
that satisfy the conditions \refeq{wachstum}--\refeq{commute}.
It is immediate from \refeq{abfall} that
for elements of the form $ \myu \in \R(\A^\pp) $ with $ \pp \le \ppz $, we have
$ \norm{ S_\para \myu  } = \Landauno{\para^\pp} $ as $ \para \to 0 $.
The following theorem provides a similar
decay property of 
$ S_\para \myu $ as $ \para \to 0 $
for elements $ u $ satisfying a mixed source condition \refeq{smooth-cond}, which in fact serves as the fundamental ingredient for the results presented below.
Its proof utilizes \cosegr theory, in particular the Hille--Yoshida theorem. The theorem is an extention of \cite[Lemma 11]{Plato_Hofmann[22]}, where 
the case $ \nu = 1 $
is considered.
\begin{theorem}
\label{th:low-order-rate}
Under the mixed smoothness condition \refeq{smooth-cond}, we have
$$
\norm{ S_\para \myu  } \le \ccdefref{lor_a} \norm{w} {\para^p \logpowinv{\tfrac{1}{\para}}}
\for 0 < \para < 1,
$$
where $ w $ is given by \refeq{log-source}, and $ \ccref{lor_a} > 0 $ denotes some finite constant which may depend on $ \lambda $ considered in
\refeq{log-source}.
\end{theorem}
\begin{proof}
There holds $ \norm{\A^q} \le C e^{\omega q} $ for $ q \ge 0 $,
where $ \omega = \log \norm{A} $ (if $ A \neq 0$), and $ C > 0 $ denote a suitable constant. 
This estimate follows easily, e.g., from the interpolation inequality \refeq{interpol2} 
applied with $ 0 < \pp < 1 = \qq $, the semigroup property~\refeq{semigroup-property} as well as the submultiplicativity of the norm of products of operators.
Thus each real $ \lambda > \omega $ belongs to the resolvent set of the operator $ \log \A: \RA \supset \domain(\log \A) \to \RA $, i.e.,
$ (\lambda I - \log \A)^{-1}: \RA \to \RA $ exists and defines a bounded operator,
\cf~\cite[Theorem 5.3, Chapter 1]{Pazy[83]}.
Since
$$ \R((\lambda I - \log \A)^{-\nu}) = \domain ((\lambda I - \log \A)^\nu) = \domain( \logpowa), $$
we may represent $ u $ as
$$ u = \A^\pp(\lambda I - \log \A)^{-\nu} w, $$
%
for some $ w \in \RA $. 
As a preparation for the considerations below, we now introduce the finite number $ \ppo $, which is defined by $ \ppo = \ppz $, if the saturation $ \ppz $ from \refeq{abfall} is finite, 
and otherwise choose $ \ppo > \pp $ arbitrary but finite.
Since (cf.~\cite[proof of Theorem 5.3, Chapter 1]{Pazy[83]})
$$  (\lambda I - \log \A)^{-\nu }w 
= \tfrac{1}{(\nu-1)!} \int_0^\infty q^{\nu-1} e^{-\lambda q} \A^q w \, dq, $$
we have
\begin{align}
\label{eq:spara-decomp}
S_\para u = \int_0^\infty h(q) \, dq = 
\int_0^{\ppo-p} h(q)  \, dq + \int_{\ppo-p}^\infty h(q)  \, dq  =:
y_1 + y_2,
\end{align}
where
%
\begin{align*}
h(q) =  \tfrac{1}{(\nu-1)!} q^{\nu-1} e^{-\lambda q} S_\para \A^{p+q} w ,
\quad q > 0.
\end{align*}
%
The element $ y_1 $ in \refeq{spara-decomp} can be estimated as
$ \normix{y_1} \le  \tfrac{\ccdefref{lor_b}}{(\nu-1)!} \norm{w} \int_0^{\ppo-p} q^{\nu-1}  \para^{p+q}\, dq $, for some finite constant $ \ccref{lor_b} > 0 $,
which follows from \refeq{abfall} and the corresponding boundedness assumption on the involved coefficient $ e_\pp $ introduced in \refeq{abfall}.
%
%
Repeated partial integration finally gives the estimate
$ \tfrac{1}{(\nu-1)!} \int_0^{\ppo-p} q^{\nu-1} \para^{p+q}\, dq 
\le \para^p \frac{1}{\modul{ \log \para }^\nu} $ for $ 0 < \para < 1 $.
%
The element $ y_2 $ can be written as follows,
\begin{align*}
y_2 =  \tfrac{1}{(\nu-1)!} \int_{\ppo-p}^\infty  q^{\nu-1} e^{-\lambda q} S_\para \A^{\ppo} \A^{q-(\ppo-p)} w \, dq,
\end{align*}
and thus we can estimate as follows:
\ccdef{lor_c}
\begin{align*}
\normix{y_2} & \le  \ccref{lor_c} \norm{w} \int_{\ppo-\pp}^\infty  q^{\nu-1} e^{-\lambda q} \para^{\ppo} e^{\omega(q-(\ppo-p))} \, dq
\\
& = \ccref{lor_c} \norm{w} e^{-\omega(\ppo-p)} \int_{\ppo-p}^\infty  q^{\nu-1} e^{-(\lambda-\omega)q} \, dq \  \para^{\ppo},
\end{align*}
for some finite constant $ \ccref{lor_c} > 0 $,
where the latter integral is finite and independent of $ \para $. This completes the proof.
\end{proof}
%
%
%
%
\section{Parameter choice strategies}
\label{parameterchoice}
Let $ A :\ix \to \ix $ be a bounded linear operator of \postype.
In what follows, we consider the regularization of the equation $ A u = \fst $
by some \regscheme, with 
some initial guess $ \ubar \in \ix $. Below, we assume that there exists a solution $ \ust \in \ix $ which satisfies condition \refeq{smooth-cond} with $ \myu = \ust - \ubar $, \ie
\begin{align}
\label{eq:smooth-cond-b}
\ubar - \ust = A^\pp v \textup{ for some } v \in \domain(\logpowa),
\quad 0 \le p < \pp_0, \quad \nu \in \naza.
\end{align}
%
%
Condition \refeq{smooth-cond-b} may be considered as a mixed smoothness condition since both the operators $ A^\pp $ and $ \log A $ are involved. Note that the case $ \pp = 0 $, this is low order smoothness, is covered by the analysis.
We allow here mixed smoothness conditions since the case $ p > 0 $ does not require much additional effort.
The notation mixed smoothness is utilized in \cite{Klinkhammer[24]}.

\subsection{A priori parameter choice strategies}
\ccdef{apriori}%
%
For the setting considered in the beginning of the present Section
\ref{parameterchoice}, we consider the following a~priori parameter choice,
\begin{align}
\label{eq:apriori}
\pardel = \ccref{apriori} \delta^{\frac{1}{\pp+1}}\logpow{\frac{\mynu}{\pp+1}}{\tfrac{1}{\delta}},
\qquad 0 < \delta  \le \delta_0 ,
\end{align}
for some constants $ 0 < \delta_0 < 1 $ and $ \ccref{apriori} > 0 $.
%
%
\begin{theorem}[A priori parameter choice]
\label{th:apriori}
Let the conditions stated in the beginning of Section
\ref{parameterchoice} be satisfied,
including the mixed smoothness condition \refeq{smooth-cond-b}.
Then the 
a priori parameter choice \refeq{apriori} yields
\ccdef{cnob}
%
\begin{align}
\label{eq:apriori-esti}
\norm{\udel -\ust} \le \ccref{cnob} \dw
\delta^{\frac{\pp}{\pp+1}}\logpow{-\frac{\mynu}{\pp+1}}{\tfrac{1}{\delta}},
\quad 0 < \delta \le \delta_0.
\end{align}
%
In \refeq{apriori-esti}, the notation
$ \dw \defeq \max\{\norm{w},1\} $ is used, with
$ w $ corresponding to the source representation \refeq{log-source}. 
In addition, 
$ \ccref{cnob} > 0 $
denotes some finite constant that is independent of $ \delta $ and $ w $.
\end{theorem}
\proof
From \refeq{upardel-def} we obtain
\begin{align*}
\upardel - \ust = S_\para (\ubar-\ust) - R_\para( \A \ust - \fdelta)
\for \para > 0,
\end{align*}
and thus, \cf
Theorem \ref{th:low-order-rate} and \refeq{wachstum},
\begin{align}
\norm{\upardel - \ust}
& \le \norm{S_\para(\ubar-\ust)} + \norm{R_\para( \A \ust - \fdelta)}
\nonumber \\
& \le
\ccref{lor_a}
\norm{w}\para^\pp \logpowinv{\tfrac{1}{\para}} + \cst\mfrac{\delta}{\para}
\for 0 < \para < 1.
\label{eq:regerror_basis_esti}
\end{align}
%
For $ \para = \pardel $, the second term in \refeq{regerror_basis_esti} clearly has the same order as the upper bound in \refeq{apriori-esti}. 
For an estimation of the first term in \refeq{regerror_basis_esti}, we 
use the notation
$ \mychi{q}{\mu}(t) = t^q \kla{\mylogb{t}}^{\mu}, \ 0 < t < 1 $, from \refeq{chi-def} in the appendix below.
We then have
\begin{align*}
\pardel = \ccref{apriori} \mychi{\frac{1}{\pp+1}}{\frac{\nu}{\pp+1}}(\delta)
\asymp 
\mychiinv{\pp+1}{-\nu}(\delta),
\end{align*}
%
\cf~item \ref{it:low_oder_prep_b} in the appendix presented in Section \ref{appendix} below.
For the meaning of $ \asymp $, we also refer to the appendix. 
Thus we have
%
$ \mychi{\pp+1}{-\nu}(\pardel) \asymp \delta $ for $ \delta > 0 $ sufficiently small,
\cf~items \ref{it:low_oder_prep_b} and \ref{it:low_oder_prep_c} in the appendix.
This finally results in
\begin{align*}
\pardel^\pp \logpowinv{\tfrac{1}{\pardel}} 
\asymp \frac{\delta}{\pardel} =
\ccref{apriori}^{-1}
\delta^{\frac{\pp}{\pp+1}}\logpow{-\frac{\mynu}{\pp+1}}{\tfrac{1}{\delta}}
\end{align*}
for $ \delta > 0 $ sufficiently small,
\cf~item \ref{it:low_oder_prep_a} in the appendix below.
%
Therefore, the first term in \refeq{regerror_basis_esti} also has the same order as the upper bound in \refeq{apriori-esti}.

We thus have obtained the desired error rate 
for $ 0 < \delta \le \delta_1 $,
with $ \delta_1 > 0 $ chosen sufficiently small.
In the degenerated case $ \delta_1 < \delta \le \delta_0 $, the given error estimate follows both from
$ \norm{\udel -\ust} 
\le \cpz \norm {\ubar-\ust} + \cst \lfrac{\delta}{\pardel} = \Landauno{\dw} $,
\cf \refeq{regerror_basis_esti},
and \refeq{abfall} for $ \pp = 0 $,
and the fact that
$ \delta^{\fracb{\pp}{\pp+1}}\logpow{-\fracb{\mynu}{\pp+1}}{\tfrac{1}{\delta}} $ is bounded away from 0.
\proofend
%

\noindent
It is clear from its proof that the statement of Theorem \ref{th:apriori} remains valid for any a priori parameter choice $ \pardel \asymp \delta^{\fracb{\pp}{\pp+1}}\logpow{-\fracb{\mynu}{\pp+1}}{\tfrac{1}{\delta}} $ as $ \delta \to 0 $. 
%
%

As a corollary of Theorem \ref{th:apriori}, we obtain the following interpolation inequality for mixed smoothness. It can be utilized for the discrepancy principle considered in the following section but may also be of independent interest.
We continue to assume that $ \A: \ix \to \ix $ is a bounded linear operator of \postype on a Banach space $ \ix $. 
\begin{theorem}[Interpolation inequality for mixed smoothness]
\label{th:loginterpol}
Let  $ u \in \ix $ be representable in the form
$ u = A^\pp v $ with $ v \in \domain(\logpowa) $, for some $ \pp \ge 0 $ and some integer $ \nu \ge 1 $.
In addition, let $ \norm{\A u} \le \delta_0 $, with
$ 0 < \delta_0 < 1 $ fixed.
Then we have
%
\begin{align}
\label{eq:loginterpol}
\norm{u} \le  \ccref{cnob} \dw
 \norm{Au}^{\frac{\pp}{\pp+1}}\logpow{-\frac{\mynu}{\pp+1}}{\tfrac{1}{\norm{Au}}},
\end{align}
where the numbers $  \ccref{cnob} $ and $ \dw $
are taken from Theorem \ref{th:apriori}.
\end{theorem}
\proof Follows immediately from Theorem \ref{th:apriori}, applied with $ \ubar = \fdelta = 0, \, \ust = \myu $ and $ \delta = \norm{Au} $, and by utilizing some
\regschemeb
with saturation $ \ppz > \pp $. Note that the considered setting means $ \upardel \equiv 0 $.
\proofend

%
%
\subsection{The discrepancy principle}
\label{noise_diskrepanz_par}
In what follows, the discrepancy principle as a parameter choice is considered
for a \regscheme,
this time with saturation $ \ppz > 1 $. Below, 
convergence rates are proven that are similar to those obtained for a priori parameter choices.
%

For method \refeq{upardel-def}, we shall choose the regularization parameter 
by means of the norm of the defect
$$ \rpardel \defeq A\upardel - \fdel $$
and the error level $ \delta $. For further details, let again
$ S_\para =  I - R_\para A $ and let the main conditions
\refeq{abfall}--\refeq{commute} be satisfied, with saturation $ p_0 > 1 $.
If $ \ust \in \ix $ solves equation~\refeq{maineq}, then we have
$$ \rpardel = S_\para \kla{A\ubar - \fdel} =
S_\para A\kla{\ubar - \ust} +
S_\para\kla{A\ust - \fdel}, $$
and estimate \refeq{abfall}
applied with $ \pp = 0 $ then
implies
\begin{equation}
\modul{ \norm{ \rpardel } - \norm{ S_\para A\kla{\ubar - \ust} } } \le \czer \delta.
\label{eq:dreieck2}
\end{equation}
This combined with \refeq{abfall} for $ \pp = 1 $ yields
$$ \limsup_{\para \to 0} \, \norm{ \rpardel } \le
\czer \delta, $$
which is one step towards applicability of the a posteriori parameter choice
considered in Algorithm~\ref{th:disrepancy_principle} below.
Another ingredient is the condition
\begin{equation}
\para \mapsto S_\para u  \text{ is a continuous function on }
\inset{ \ 0 < \para < \infty } \text{ for each } u \in \ix
\label{eq:H_t-stetig}
\end{equation}
which holds for \lavmetit, for example.
%
%
This continuity condition
guarantees that the set of parameters $ \para_\delta > 0 $ considered in 
Algorithm~\ref{th:disrepancy_principle} below
is non-void.
%
%

We are now in a position to present a discrepancy principle
for the considered regularization method.
As a further preparation, 
we extent the notation in \refeq{rbeta} by additionally considering the vanishing operator
$ R_\infty = 0 : \ix \to \ix $, which allows to consider 
in \refeq{upardel-def} the element
$ \upardel[\infty] = \ubar $. 
\begin{algorithm}[Discrepancy principle]
\label{th:disrepancy_principle}
Fix positive constants
$ b_0, \ b_1 $ with $ b_1 \ge b_0  > \czer $,
with $ \czer $ from \refeq{abfall}.
\begin{myenumerate_indent}
\item
If $ \norm{ A \ubar - \fdelta } \le b_1 \delta $
then set $ \pardel = \infty $, \ie
$ \upardel[\infty] = \ubar $.

\item
If otherwise $ \norm{ A \ubar - \fdelta  } > b_1 \delta $
then choose the parameter
$ 0 < \para_\delta < \infty $ such that the following
conditions are satisfied,
$$ b_0 \delta \le \norm{  \rpardel[\pardel] }
\le b_1 \delta.
\remarkend
$$
\end{myenumerate_indent}
\end{algorithm}
%
\noindent
The parameter $ \pardel $ depends also on the perturbed \rhs $ \fdel $
and hence is an a~posteriori parameter choice. For notational purpose, this dependence is not further indicated, however.
Note that the continuity condition \refeq{H_t-stetig} and the conditions
\refeq{abfall}, \refeq{commute} with saturation $ \ppz > 1 $ guarantee that the set of parameters $ \para_\delta > 0 $ considered in
Algorithm \ref{th:disrepancy_principle}
is non-void.
Note also that (b) in Algorithm \ref{th:disrepancy_principle} is the standard case, while
(a) corresponds to a degenerated case.

The statement \refeq{rate2} in the theorem below shows that the discrepancy principle,
applied to
a \regscheme
with saturation $ \ppz > 1 $,
achieves the same rates as a priori parameter choices. Note that,
due to the condition $ \ppz > 1 $, the theorem
is applicable for \lavmetit for $ m \ge 2 $ only.
We continue to assume that $ \A: \ix \to \ix $ is a bounded linear operator of \postype on a Banach space $ \ix $. 
\begin{theorem}[A posteriori parameter choice]
%
%
Consider a \regschemec
with saturation $ \pp_0 > 1 $.
Let $ \ubar - \ust = A^\pp v $ for some $ v \in \domain(\logpowa) $,
for some real number $ 0 \le  \pp < \pp_0 - 1 $
and some integer $ \nu \ge 1 $, and let $ 0 < \delta_0 < 1 $.
Then
there exist finite constants
$ \ccdefref{apost_a}, \, \ccdefref{apost_b} > 0 $, 
such that for any $ 0 < \delta \le \delta_0 $, we have
%
\begin{align}
\norm{ \upardel[\pardel] - \ust }
& \le \ccref{apost_a} \delta^{\frac{\pp}{\pp+1}}\logpow{-\frac{\mynu}{\pp+1}}{\tfrac{1}{\delta}},
\label{eq:rate2} \\
\para_\delta & \ge
\ccref{apost_b} \delta^{\frac{1}{\pp+1}}\logpow{\frac{\mynu}{\pp+1}}{\tfrac{1}{\delta}},
\label{eq:pardelrate2}
\end{align}
where the parameter $ \pardel $ is chosen according to the discrepancy principle introduced above, cf.~Algorithm \ref{th:disrepancy_principle}.
%
\label{th:disparameter}
\end{theorem}
%
%
\proof
We verify first the lower bound for the parameter $ \pardel $ in
\refeq{pardelrate2},
and for this purpose, \mywlog we may assume that
$ 0 < \pardel < 1 $ holds.
From estimate \refeq{dreieck2}
and Theorem~\ref{th:low-order-rate} we then find
\ccdef{apost_c}
\begin{equation}
\kla{b_0 - \cp[0] } \delta
\le 
\norm{ S_{\para_{\delta}} A\kla{\ubar - \ust} }
\le
\ccref{apost_c} \cdott \pardel^{\pp+1} \logpowinv{\tfrac{1}{\pardel}} = \ccref{apost_c} \mychi{\pp+1}{-\mynu}(\pardel), 
\label{eq:lowerbound}
\end{equation}
%
%
%
for some finite constant $ \ccref{apost_c} > 0 $, and \mywlog we may assume
$ \ccref{apost_c} \ge b_0-\cpz $.
For the meaning of the notation $\mychi{\pp+1}{-\mynu} $, we refer to \refeq{chi-def} in the appendix below.
A reformulation of \refeq{lowerbound} gives,
utilizing items \ref{it:low_oder_prep_b} and \ref{it:low_oder_prep_c}
of the appendix below, 
%
%
\ccdef{apost_d}
\ccdef{apost_e}
\ccdef{apost_f}
\begin{align*}
\pardel & \ge \mychiinv{p+1}{-\mynu}(\ccref{apost_d} \delta)
\ge \ccref{apost_e} \mychi{\frac{1}{p+1}}{\frac{\mynu}{p+1}}(\ccref{apost_d} \delta)
\\
& \ge \ccref{apost_f} \mychi{\frac{1}{p+1}}{\frac{\mynu}{p+1}}(\delta)
= \ccref{apost_f} \delta^{\frac{1}{\pp+1}}\logpow{\frac{\mynu}{\pp+1}}{\tfrac{1}{\delta}},
\end{align*}
where $ \ccref{apost_d} = \frac{b_0-\cpz}{\ccref{apost_c}} \le 1 $, and 
$ \ccref{apost_e} $ and $ \ccref{apost_f} $ denote appropriately chosen finite positive constants.
Estimate \refeq{pardelrate2} is thus verified.

To prove estimate \refeq{rate2} for
the error $ \norm{ \upardel[\pardel] - \ust } $,
we start with
%
the basic estimate, \cf~\refeq{regerror_basis_esti},
\begin{equation}
\norm{ \upardeldel - \ust }
\le \norm{ S_{\pardel}\kla{\ubar - \ust}} + \cst \mfrac{\delta}{\pardel}.
\label{eq:mainineq2}
\end{equation}
The second term on the \rhs of \refeq{mainineq2} can be properly estimated by using estimate \refeq{pardelrate2},
and below we consider the first term on the \rhs of \refeq{mainineq2}. Below, it will be estimated by means of the interpolation inequality for mixed smoothness considered in
Theorem \ref{th:loginterpol}.
For this purpose, we first note that estimate \refeq{dreieck2} yields
\begin{equation}
\norm{ S_{\pardel} A\kla{\ubar - \ust} }
\le \kla{\cpz + b_1} \delta.
\label{eq:Stestimate}
\end{equation}
%
%
As a further preparation we note that, for any $ \para > 0 $, the element
$ S_{\para} \myu $ satisfies a mixed source source condition
\refeq{smooth-cond} if $ \myu \in \ix $ does. In addition, the former source element that does not exceed, up to a constant factor, the norm of the latter source element.
In fact, the operator $ S_\para $ commutes with $ A $ by assumption, and thus it commutes with each fractional power $ A^\pp, \, \pp > 0 $. For each 
$ v \in \domain (\loga) $ we thus have
$ S_\para v \in \domain(\loga) $ and $ S_\para (\loga) v = (\loga) S_\para v $, and the claim on the mixed source representability of $ S_{\para} \myu $ is now easily obtained.

We thus may apply the interpolation inequality for mixed smoothness \refeq{loginterpol} 
to the element $ \myu= S_{\pardel} \kla{\ubar - \ust } $ for any $ 0 < \delta \le \delta_1 $, with $ 0 < \delta_1 < \tfrac{1}{\cpz + b_1} $ arbitrary but fixed,
\cf estimate \refeq{Stestimate}. 
We finally obtain the following estimate:
\ccdef{apost_z}%
\ccdef{apost_g}%
\begin{align*}
\norm{ S_{\pardel} \kla{\ubar - \ust } }
& \le
\ccref{apost_z} \norm{ AS_{\para_\delta} \kla{\ubar - \ust} }^{\frac{\pp}{\pp+1}}
\cdot \logpow{-\frac{\mynu}{\pp+1}}{\tfrac{1}{\norm{AS_{\para_\delta} \kla{\ubar - \ust}}}}
\\
& \le
\ccref{apost_z} \kla{\kla{\cpz + b_1} \delta}^{\frac{\pp}{\pp+1}}\cdot \logpow{-\frac{\mynu}{\pp+1}}{\tfrac{1}{\kla{\cpz + b_1} \delta}}
\\
& \le
\ccref{apost_g} \delta^{\frac{\pp}{\pp+1}}\cdot \logpow{-\frac{\mynu}{\pp+1}}{\tfrac{1}{\delta}},
\end{align*}
where $ \ccref{apost_z}, \, \ccref{apost_g} $ are some constants that may depend on the source element associated with $ \ubar - \ust $ but are independent of $ \delta $.
This completes the proof of the first statement \refeq{rate2} of the theorem for the case
$ 0 < \delta \le \delta_1 $.
In the case $ \delta_1 < \delta \le \delta_0 $, the statement follows
similar to the proof of Theorem \ref{th:apriori}: the error
$ \norm{\udel -\ust} $ stays bounded while
$ \delta^{\fracb{\pp}{\pp+1}}\logpow{-\fracb{\mynu}{\pp+1}}{(\lfrac{1}{\delta})} $ is bounded away from 0.
\proofend
We note that the two constants $ \ccref{apost_a}, \, \ccref{apost_b} $ in Theorem
\ref{th:disparameter}
may depend on the source element corresponding to
the initial error $ \ubar - \ust $.
\section{A low order smoothness example}
\label{loworder_example}
In the present section, we present an example of a function that satisfies a \losoc with respect to the integration operator on a space of continuous functions.
For related considerations in $L^p$-spaces, we refer to
\cite[Remark 2.5 on p.~52]{Samko_Kilbas_Marichev[06]}.
We only present basic ideas and leave the details to the reader.

Let  
$ \co := \inset{u \in C[0,1]:u(0) = 0 }$ and
\begin{align}
\klaf{Ju}(x) = \ints{0}{x}{ u(\myxi) }{ d\myxi } \qquad (0 \le x \le 1, \ u \in \co).
\label{eq:inteqop}
\end{align}
For this setting, the integration operator $ J: \co \to \co $ defines an operator \ofpostype with dense range, 
see Example \ref{th:postype-examples}. Below, we show that the continuous function
\begin{align}
\label{eq:ulog}
u(\myxi) = \left\{ 
\begin{array}{rl} (-\mylog c \myxi)^{-\kappa}, & 0 < \myxi \le 1 , \\
0, & \myxi = 0,
\end{array} \right.
\end{align}
belongs to $ \domain(\logj) $ 
for $ 0 < c < 1 $ and $ \kappa > 1 $ arbitrary but fixed. 
We start with an auxiliary result for \cosegrs on general Banach spaces.
\begin{proposition}
\label{th:semigroup-prop}
Let $ T(t), \, t \ge 0, $ be a \cosegr on a Banach space $ \ix $
 with infinitesimal generator $ A : \ix \supset \domain(A) \to \ix $, and let $ u \in \ix $.
Moreover, let $ T(t) u $ be differentiable at $ t = 1 $, and let $ T(1) $ be a one-to-one operator.  Then $ u \in \domain(A) $ holds if and only if $ \frac{d}{dt} T(t) u |_{t = 1} \in \R(T(1)) $.
\end{proposition}
\begin{proof}
This is a straightforward generalization of \cite[proof of Theorem 23.6.1]{Hille_Phillips[57]} to general \cosegrs. 
\end{proof}

\noindent
Below, we apply this proposition to the semigroup of fractional powers of the integration operator as considered in \refeq{inteqop}. Those fractional powers in fact are the Abel integral operators considered in Example \ref{th:postype-examples},
$ T(\myalpha) = J_\myalpha $ for $ \myalpha > 0 $ and thus in particular $T(1) = J $. It turns out that the conditions of Proposition~\ref{th:semigroup-prop} are satisfied. We thus have
\begin{align}
\label{eq:abel_derivative}
\tfrac{d}{d\myalpha} \klaf{J_\myalpha u}_{\myalpha=1} = Su -c J u \for u \in \co,
\end{align}
where $ c = \Gamma^\prime(1) $ denotes Euler's constant, and 
\begin{align}
\label{eq:voltlog}
\klaf{Su}(x) = 
\ints{0}{x}{ \log(x-\myxi) u(\myxi) }{ d\myxi } \qquad (0 < x \le 1, \ u \in \co),
\end{align}
cf.~again \cite[Theorem 23.6.1]{Hille_Phillips[57]} for a similar setting in other spaces.
The kernel of the linear Volterra operator $ S: \co \to \co $
in \refeq{voltlog} is weakly singular in the sense of
\cite[Section 2.5]{Kress[14]}, and 
thus $ S $ is a compact operator, see \eg \cite[Theorem 2.29]{Kress[14]}.  Since 
$ \R(J) = \inset{u \in C^1[0,1] : u(0) = u^\prime(0) = 0 }$, it follows from
the identity \refeq{abel_derivative}
and Proposition \ref{th:semigroup-prop} that any $ u \in \co $ satisfies 
$ u \in \domain(\logj) $ if and only if both $ Su \in C^1[0,1] $ and $ (Su)^\prime(0) = 0 $ holds.

Below we show that the latter statements on the function $ Su $ indeed hold for the function $ u $ from
\refeq{ulog}. For this purpose, we consider
%
%
%
approximations of $ Su $ utilizing cutting functions
as follows, \cf \cite[Sections 3.3 and 4]{Vainikko[07]},
%
\begin{align*}
v_n(x) \defeq \ints{0}{x}{ e(n (x-y)) \log(x-\myxi) u(\myxi) }{ d\myxi } \forsh 0 < x \le 1, \
n \in \naza,
\end{align*}
where $ e \in C^1[0,\infty) $ is some cutting function satisfying $ e(r) = 0 $ for $ 0 \le r \le \tfrac{1}{2} $, $ e(r) = 1 $ for $ r \ge 1 $, and $ 0 \le e(r) \le 1 $ for $ \tfrac{1}{2} <  r < 1 $. 
It follows from the techniques considered in \cite[Sections 3.3 and 4]{Vainikko[07]}
that the functions $ v_n $ are continuous differentiable on the interval $ (0, 1] $, with
\begin{align*}
\prim{v}_n(x) = \ints{0}{x}{ e(n (x-y)) \log(x-\myxi) u^\prime(\myxi) }{ d\myxi } \forsh 0 < x \le 1, \ n \in \naza.
\end{align*}
Here we make use of the fact that the derivative $ \prim{u}(\myxi) =
\kappa \lfrac{(-\mylog c \myxi)^{-\kappa-1}}{\myxi} $ for $ 0 < \myxi \le 1 $ satisfies
$ \prim{u} \in C(0,1] \cap L^1(0,1) $.
Moreover, we have $ v_n \to Su $ and $ v_n^\prime \to w $ uniformly 
on each interval of the form $ [\varepsilon, 1] $, with $ 0 < \varepsilon \le 1 $ arbitrarily small,
where
\begin{align}
\label{eq:suprime}
w(x) \defeq \ints{0}{x}{ \log(x-\myxi) u^\prime(\myxi) }{ d\myxi } \forsh 0 < x \le 1.
\end{align}
%
%
Here we use the fact that the convolution kernel
$ k(\myxi) \defeq \log \myxi $ in \refeq{suprime}
is weakly singular.
Thus, $ Su \in C^1(0,1] $ and $ \prim{(Su)} = w $ on $ (0,1] $. This result follows also directly from the techniques utilized in
\cite[Lemma 3.3]{Vainikko[93]} or
\cite[Lemma 1]{Vainikko_Pedas[81]}.
Note that the integral \refeq{suprime} exists, which follows, e.g., from Fubini's theorem, \cf also \cite[Theorem~9.5.1]{Edwards[65]}.

Finally, elementary calculations show that $ w(x) \to 0 $ holds as $ x \to 0 $,
since $ \kappa > 1 $.
Thus the function $ Su $ is also differentiable at $ x = 0 $, with 
$ \klaf{Su}^\prime(0) = 0 $ as claimed.
%
%
%
%
%
\section{Conclusions and outlook}
In this work, we have considered the regularization of linear ill-posed problems $ Au = \fst $ with \postype operators $ A: \ix \to \ix $ in a Banach space $ \ix $.
Our focus is on the stable recovery of solutions $ \ust \in \ix $ that allow a logarithmic source representation $ \ust = A^\pp v \textup{ for some } v \in \domain(\logpowa) $, with some $ \pp \ge 0 $ and some integer $ \nu \ge 1 $.
Convergence rates for a class of parametric regularization schemes have been deduced that include the iterated version of Lavrentiev's method and the method of the abstract Cauchy problem. The presentation includes both a priori and a posteriori parameter choice strategies.

Further projects in this direction not considered in this paper are the implementation of numerical experiments, the consideration of iterative methods as regularization schemes, and, last but not least, smoothness conditions of the form
$ u^\dagger = \logpowa A^\pp v $.

\section{Appendix}
\label{appendix}
In the analysis of mixed order rates, the functions
\begin{align}
\mychi{q}{\pm\mu}(t) & = t^q \kla{\mylogb{t}}^{\pm\mu}, \quad 0 < t < 1
\qquad \kla{q \ge 0, \, \mu > 0},
\label{eq:chi-def}
\end{align}
are of major relevance. Below we state some elementary properties of those functions, cf.~also \cite{Hohage[00],Mair[94],Tautenhahn[98]}.
%
\begin{myenumerate_indent}
\item
\label{it:low_oder_prep_a}
The function $ \mychi{q}{-\mu} $ is strictly increasing on the interval $ \kla{0,1}  $, with $ \mychi{q}{-\mu}(t) \to 0 $ as $ t \to 0 $.


\item
\label{it:low_oder_prep_b}
For $ q > 0 $,
the inverse function $ \mychiinv{q}{-\mu}: (0,\infty) \to \reza $
satisfies
%
$ \mychiinv{q}{-\mu}(s) \sim
q^{-\mu/q} s^{1/q} \kla{\mylogb{s}}^{\mu/q} $ as $ s \to 0 $.
%
This implies in particular that, for each fixed $ 0 < s_0 < 1 $, we have
\begin{align*}
\mychiinv{q}{-\mu}(s) \asymp s^{\frac{1}{q}} \kla{\mylogb{s}}^{\frac{\mu}{q}}
= \mychi{\frac{1}{q}}{\frac{\mu}{q}}(s),
\quad 0 < s \le s_0.
\end{align*}

\item
\label{it:low_oder_prep_c}
For each constant $ \kappa > 0 $, we have
$ \mychi{q}{\pm\mu}(\kappa t) \sim \kappa^q \mychi{q}{\pm\mu}(t) $ as $ t \to 0 $,
and thus in particular
\begin{align*}
\mychi{q}{\pm\mu}(\kappa t) \asymp \mychi{q}{\pm\mu}(t), \quad 0 < t \le t_0,
\end{align*}
for each $ t_0 < \min\{1,1/\kappa\} $ fixed.

%

%
%

%
%
%
%
\end{myenumerate_indent}
Here, for two positive, real-valued functions $ f, g: (0,t_0) \to \reza $, the notation $ f(t) \sim g(t) $ as $ t \to 0 $ means $ f(t)/g(t) \to 1 $ as $ t \to 0 $. In addition, $  f(t) \asymp g(t) $ for $ t \in I \subset (0,t_0) $ means that there are finite positive   constants $ a_1, a_2 $ such that
$ a_1 f(t) \le g(t) \le a_2 f(t) $ for $ t \in I $.

\bigskip
\bigskip\noindent
\textbf{Acknowledgements.} The author is grateful to B.~Hofmann (Technical University Chemnitz) for the motivating impulses to pursue logarithms of operators in Banach spaces, which have made this work possible. 

This paper was supported by the German Research Foundation (DFG) under grant PL~182/8-1.
\bibliography{../datenbanken/illposed,../datenbanken/fracpower,../datenbanken/standard,../datenbanken/volterra}

\end{document}

%% file: macros.txt
\newcommand{\para}{\alpha}

\newcommand{\pardel}{\para_\delta}

\newcommand{\ix}{\mathcal{X}}

\newcommand{\pp}{p}
\newcommand{\ppz}{\pp_0}
\newcommand{\ppo}{\pp_1}
\newcommand{\qq}{q}

\newcommand{\fdelta}{f^\delta}
\newcommand{\reza}{ \mathbb{R} }
\newcommand{\naza}{ \mathbb{N} }
\newcommand{\Landau}{\mathcal{O}}
\newcommand{\refeq}[1]{(\ref{eq:#1})}

\newcommand{\norm}[1]{\Vert \hspace{0.4mm} #1 \hspace{0.4mm} \Vert}

\newenvironment{myenumerate}{%
\begin{list}{(\alph{enumcount})}
{\setcounter{enumcount}{1}\usecounter{enumcount}
\setlength{\topsep}{1mm}
\setlength{\itemsep}{0mm}
\setlength{\labelwidth}{0mm}
\setlength{\labelsep}{1mm}
\setlength{\itemindent}{1mm}
\setlength{\leftmargin}{0mm}
}}{\end{list}}

\newcommand{\myu}{u}

\newcommand{\ust}{u^\dagger}

\newcommand{\fst}{f}
\newcommand{\fdel}{f^\delta}

\newcommand{\upardeldel}{u_{\pardel}^\delta}

\newcommand{\tinybullet}{{\tiny \raisebox{0.6mm}{$ \bullet $}}}

\newcommand{\cdott}{\hspace{0.4mm}}

\newcommand{\kla}[1]{(#1)}
\newcommand{\klaf}[1]{\kla{#1}}
\newcommand{\mfrac}[2]{\dfrac{\mbox{\footnotesize \raisebox{-0.5mm}{$#1$}}}%
{\mbox{\footnotesize \raisebox{0.8mm}{$#2$}}}}
\newcommand{\proofend}{\qquad \endproof}

\newcommand{\prim}[1]{#1^{\prime}}
\newcommand{\Landauno}[1]{\Landau\kla{#1}}

\newcommand{\as}{\quad \text{as} \ \ }

\newcommand{\R}{\mathcal{R}}
\newcommand{\cf}{cf.\mbox{}\xspace}
\newcommand{\ie}{i.e.,\xspace}
\newcommand{\eg}{e.g.,\xspace}
\newcommand{\remarkend}{\quad \ensuremath{\vartriangle}}

\newcommand{\defeq}{:=}
\newcommand{\for}{\quad \text{for} \ \ }
\newcommand{\forsh}{\ \ \  \text{for} \ }

\newcommand{\ints}[4]{\mathop{\raisebox{-0.1mm}{\mbox{\Large$ \textstyle \int $}}}\nolimits_{\hspace{-1mm}#1}^{#2} \hspace{-0.2mm} #3 \, #4}

\newcommand{\rhs}{right-hand side\xspace}

\newenvironment{myenumerate_indent}{%
\begin{list}{(\alph{enumcount})}
{\setcounter{enumcount}{1}\usecounter{enumcount}
\setlength{\topsep}{1mm}
\setlength{\itemsep}{0mm}
\setlength{\labelwidth}{5mm}
\setlength{\labelsep}{2mm}
\setlength{\itemindent}{-0mm}
\setlength{\leftmargin}{7mm}
}}{\end{list}}

\newcommand{\inset}[1]{\{ \, #1 \, \}}

\newcommand{\modul}[1]{\vert \hspace{0.3mm} #1 \hspace{0.3mm} \vert}

\newcommand{\normix}[1]{\norm{#1}}

\newcommand{\lfrac}[2]{#1\slash #2}

\newcommand{\upardel}[1][\para]{{u}_{#1}^\delta}
\newcommand{\rpardel}[1][\para]{{r}_{#1}^\delta}

\newcommand{\updd}{\upardel[\pardel]}

\newcommand{\udel}{\updd}

\newcommand{\ubar}{\overline{u}}
\newcommand{\udelta}{\myu^\delta}

\newcommand{\cp}[1][p]{e_{#1}}
\newcommand{\czer}{\cp[0]}
\newcommand{\cst}{e_*}
\newcommand{\fracpower}{fractional power\xspace}
\newcommand{\fracpowers}{\fracpower{}s\xspace}
\newcommand{\domain}{\mathcal{D}}

\newcommand{\kapst}{\kappa_*}
\newcommand{\RA}{\overline{\R(\A)}}

\newcommand{\postype}{positive type\xspace}
\newcommand{\Postype}{Positive type\xspace}
\newcommand{\ofpostype}{of \postype{}\xspace}

\newcommand{\logpowinv}[1]{\kla{\log #1}^{-\mynu}}
\newcommand{\logpow}[2]{\kla{\log #2}^{{#1}}}
\newcommand{\loga}{\log \A}
\newcommand{\logj}{\log J}
\newcommand{\mylog}{\log}
\newcommand{\A}{A}
\newcommand{\logpowa}[1][\mynu]{(\log \A)^{#1}}
\newcommand{\mynu}{\nu}

\newcommand{\lavmetpur}{Lavrentiev method\xspace}
\newcommand{\lavmet}{the \lavmetpur{}\xspace}
\newcommand{\lavmetit}{the $m$-times iterated \lavmetpur{}\xspace}

\newcommand{\mychi}[2]{\chi_{#1,#2}}
\newcommand{\mychiinv}[2]{\chi^{-1}_{#1,#2}}
\newcommand{\mylogb}[1]{\log \tfrac{1}{#1}}
\newcommand{\ccdef}[1]{\refstepcounter{ccount}\label{cc#1}}%
\newcommand{\ccref}[1]{c_{\ref{cc#1}}}%
\newcommand{\ccdefref}[1]{\protect\ccdef{#1}\ensuremath{\ccref{#1}}}%
\newcommand{\cpz}{e_0}
\newcommand{\apsegr}{A^\pp, \, \pp\ge 0,}%
\newcommand{\ii}{\textup{i}}%
\newcommand{\losoc}{logarithmic source condition\xspace}
\newcommand{\myxi}{y}
\newcommand{\cosegr}{$C_0$-semigroup\xspace}
\newcommand{\cosegrs}{\cosegr{}s\xspace}
\newcommand{\co}{C_0[0,1]}
\newcommand{\mya}{a}
\newcommand{\mywlog}{w.l.o.g.\ }
\newcommand{\semidefinite}{semi-definite\xspace}
\newcommand{\los}{low order smoothness\xspace}
\newcommand{\dw}{d_w}
\newcommand{\regscheme}{scheme of the form \refeq{upardel-def}, with a regularization operator \refeq{rbeta} that satisfies the conditions \refeq{wachstum}--\refeq{commute}\xspace}
\newcommand{\regschemeb}{regularization method of the form \refeq{upardel-def}, with a regularization operator \refeq{rbeta} that obeys the conditions \refeq{wachstum}--\refeq{commute}\xspace}
\newcommand{\regschemec}{scheme of the form \refeq{upardel-def} which satisfies \refeq{wachstum}--\refeq{commute} and \refeq{H_t-stetig}\xspace}
\newcommand{\myalpha}{p}
\newcommand{\fracb}[2]{#1/(#2)}%